\documentclass[12pt]{amsart}
\usepackage{amsfonts}
\usepackage{amssymb}
\usepackage{amsmath}
\usepackage{amsthm}
\usepackage{amscd}
\usepackage[all]{xy}
\usepackage{graphicx, color}
\usepackage{url}

\allowdisplaybreaks[0]

\newtheorem{theorem}{Theorem}[section]

\newtheorem{proposition}[theorem]{Proposition}

\theoremstyle{definition}

\newtheorem{example}[theorem]{Example}
\begin{document}

\def\Im#1{
 \mathrm{Im} \hskip 0.2em #1
}

\title[Quandle and hyperbolic volume]{Quandle and hyperbolic volume}
\author{Ayumu Inoue}
\address{Department of Mathematics, Tokyo Institute of Technology, Oh--okayama, Meguro--ku, Tokyo, 152--8551 Japan}
\email{ayumu7@is.titech.ac.jp}


\begin{abstract}
We show that the hyperbolic volume of a hyperbolic knot is a quandle cocycle invariant.
Further we show that it completely determines invertibility and positive/negative amphicheirality of hyperbolic knots.
\end{abstract}

\maketitle

\section{Introduction}\label{sec:introduction}

A quandle introduced by D. Joyce \cite{Joyce1} and S. V. Matveev \cite{Matveev1} independently, is an algebraic system having a self-distributive binary operation whose definition is motivated by knot theory.
They defined the knot quandle, and showed that it completely classifies knots.
J. S. Carter et al. have developed a theory of quandle cocycle invariants in \cite{CJKLS1}.
Several useful applications of quandle homology/cohomology theory have been established; distinguishing the unknot \cite{Eisermann1}, determining non-invertibility of classical/surface knots \cite{CJKLS1, RS1, Satoh1}, and estimating the minimal triple point number of a surface knot \cite{SS1}, for examples.
However, there seems to be no conceptual understanding of quandle cocycle invariants so far.

In this paper, we would like to present such one by showing that there is a quandle cocycle invariant whose each element is $1$, $-1$ or $0$ times volume for the hyperbolic knots (Theorem \ref{thm:2_cocycle}).
Further we show that it completely determines invertibility and positive/negative amphicheirality of hyperbolic knots (Theorem \ref{thm:negative_amphichirality}, \ref{thm:invertibility}, and \ref{thm:positive_amphichirality}).

\subsection*{Acknowledgments}
The author would like to express his sincere gratitude to Professor Sadayoshi Kojima for encouraging him.
He is also grateful to Dr. Shigeru Mizushima for his invaluable comments.
This research has been supported in part by JSPS Global COE program ``Computationism as a Foundation for the Sciences''.

\section{Preliminaries}\label{sec:preliminaries}

\subsection{Knot quandle}\label{subsec:knot_quandle}

In this subsection, we briefly recall the definition of a quandle and the knot quandle.
See \cite{FR1, Joyce1, Matveev1} for examples for more details.

A {\it quandle} is defined to be a set $Q$ with a binary operation $\ast$ on $Q$ satisfying the following properties:
\begin{itemize}
\item[(Q1)] For each $x \in Q$, $x \ast x = x$.
\item[(Q2)] For each $y \in Q$, the map $\ast y : Q \rightarrow Q$ ($x \mapsto x \ast y$) is bijective.
\item[(Q3)] For each triple $x,y,z \in Q$, $(x \ast y) \ast z = (x \ast z) \ast (y \ast z)$.
\end{itemize}
For example, if we define a binary operation $\ast$ on a subset $X$ of a group $G$ closed under conjugations by
\[
a \ast b = b^{-1}ab \quad ({}^{\forall} a, b \in X)
\]
then $X$ together with $\ast$ becomes a quandle.
We call it the {\it conjugation quandle}.

Suppose that $K$ is an oriented prime knot in $S^{3}$.
It is easy to see that the set $\mathcal{Q}(K)$ of positive meridians of $\pi_{1}(S^{3} \setminus K)$, which are oriented meridians compatible with the orientation of the knot, is closed under conjugations.
The {\it knot quandle} of $K$ is defined to be its conjugation quandle.

\subsection{Quandle cocycle invariant}\label{subsec:quandle_cocycle_invariant}

In this subsection, we briefly recall the definition of a quandle cocycle invariant.
See \cite{CEGS1, CJKLS1, FRS1, FRS2, Kamada1} for examples for more details.

Let $\mathcal{F}(Q)$ be the free group on $Q$ and $\mathcal{N}(Q)$ the subgroup of $\mathcal{F}(Q)$ normally generated by
\[
 y^{-1} \hskip 0.2em x \hskip 0.2em y \hskip 0.2em (x \ast y)^{-1} \quad ({}^{\forall}x, y \in Q).
\]
We denote the quotient group $\mathcal{F}(Q) / \mathcal{N}(Q)$ by $\mathcal{G}(Q)$.

Suppose that $D$ is a diagram of an oriented knot $K$.
An {\it arc coloring} of $D$ is defined to be a map
\[
 \mathcal{A} : \{ \mathrm{arcs} \ \mathrm{of} \ D \} \longrightarrow Q
\]
satisfying the condition illustrated in the left-hand side of Figure \ref{fig:coloring} at each crossing point.
Further a {\it region coloring} of $D$ is defined to be a map
\[
 \mathcal{R} : \{ \mathrm{regions} \ \mathrm{of} \ D \} \longrightarrow Y,
\]
where $Y$ is a set equipped with a right action of $\mathcal{G}(Q)$, satisfying the condition depicted in the right-hand side of Figure \ref{fig:coloring} around each arc.
We call a pair ($\mathcal{A}$, $\mathcal{R}$) a {\it shadow coloring} of $D$, and denote by $\mathcal{S}$.

\begin{figure}[htb]
\begin{center}
\includegraphics[scale=0.5]{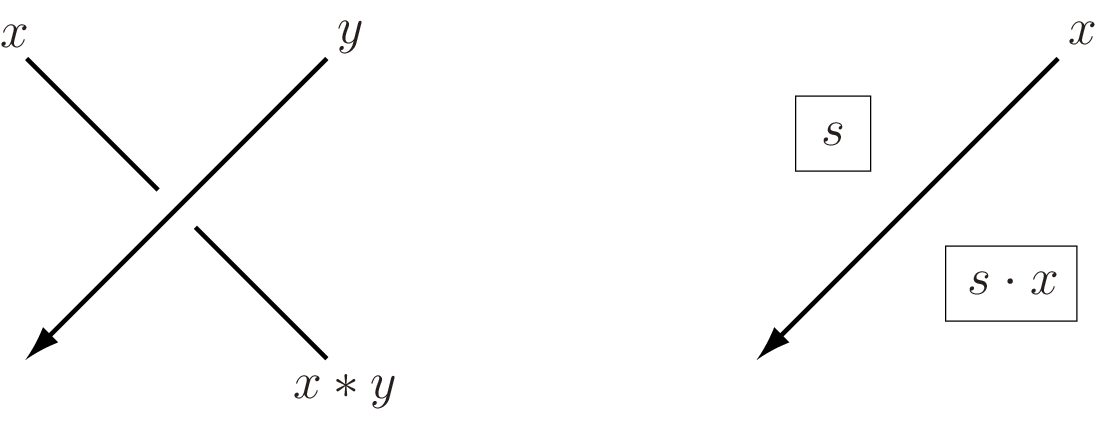}
\end{center}
\vspace{-7pt}
\caption{Rules for colorings}
\label{fig:coloring}
\end{figure}

Choose an abelian group $A$.
An $A$-valued {\it quandle $2$-cocycle} with respect to $Q$ and $Y$ is defined to be a map
\[
 \theta : Y \times Q \times Q \longrightarrow A
\]
satisfying the following conditions:
\begin{itemize}
\item[\hskip 0.1em($\mathrm{i}$)\hskip 0.2em] For each $r \in Y$ and $x \in Q$, $\theta(r, x, x) = 0$.
\item[($\mathrm{ii}$)] For each $r \in Y$ and $x, y, z \in Q$,
\begin{eqnarray*}
 & \hskip -5em \theta(r, x, y) + \theta(r \cdot y, x \ast y, z) + \theta(r, y, z) \\
 & \hskip 1em = \theta(r \cdot x, y, z) + \theta(r, x, z) + \theta(r \cdot z, x \ast z, y \ast z).
\end{eqnarray*}
\end{itemize}
For each crossing point $c$ of $D$, a Boltzmann weight of $c$ is defined as
\[
 \hskip 0.5em B(\mathcal{S}, \theta, c) = \varepsilon(c) \theta(r, x, y),
\]
where $\varepsilon(c)$ is $1$ or $-1$ depending on whether $c$ is positive or negative respectively, and $r \in Y$ and $x, y \in Q$ denote colors around $c$ as depicted in Figure \ref{fig:boltzmann_weight}.
Further we let
\[
 \Phi(\mathcal{S}, \theta) = \sum_{c \in \mathcal{C}} B(\mathcal{S}, \theta, c),
\]
where $\mathcal{C}$ denotes the set of crossing points of $D$.

\begin{theorem}[\cite{CJKLS1, FRS1, Kamada1}]
The multi-set
\[
 \{ \Phi(\mathcal{S}, \theta) \in A \mid \mathcal{S} \hskip 0.3em \mathrm{is \hskip 0.3em a \hskip 0.3em shadow \hskip 0.3em coloring \hskip 0.3em of} \hskip 0.3em D \hskip 0.3em \mathrm{with \hskip 0.3em respect \hskip 0.3em to} \hskip 0.3em Q \hskip 0.3em \mathrm{and} \hskip 0.3em Y \}
\]
does not depend on the choice of a diagram $D$ of $K$.
\end{theorem}

\noindent
We call this multi-set a {\it quandle cocycle invariant} of $K$.

\begin{figure}[htb]
\begin{center}
\includegraphics[scale=0.5]{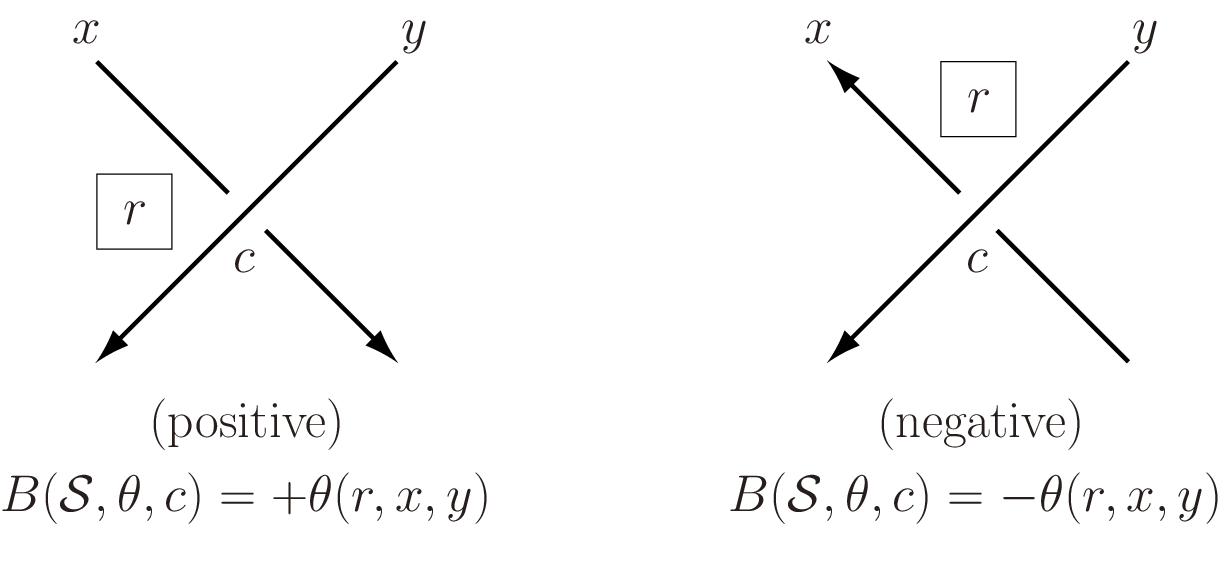}
\end{center}
\vspace{-10pt}
\caption{Boltzmann weight of a positive/negative crossing}
\label{fig:boltzmann_weight}
\end{figure}

\section{Hyperbolic volume is quandle cocycle invariant}\label{sec:hyperbolic_volume_is_quandle_cocycle_invariant}

Let $K$ be an oriented hyperbolic knot in $S^{3}$,
\[
 \Psi : \mathbb{H}^{3} \longrightarrow S^{3} \setminus K
\]
the universal covering, $p \in S^{3} \setminus K$ a base point of $\pi_{1}(S^{3} \setminus K)$, and $\widetilde{p} \in \Psi^{-1}(p)$.
Then we have a holonomy representation
\[
 \rho : \pi_{1}(S^{3} \setminus K) {\longrightarrow} \mathrm{Isom}_{+} \mathbb{H}^{3}.
\]
For each positive meridian $x \in \mathcal{Q}(K)$, remarking that $\rho(x)$ is parabolic, we denote the fixed point of $\rho(x)$ on $\partial \overline{\mathbb{H}^{3}} = S^{2}_{\infty}$ by $x_{\infty}$.

Let $q \in S^{3} \setminus K$ be a point other than $p$, and $Z$ the set of homotopy classes of paths from $p$ to $q$.
Then $Z$ admits the right action of $\mathcal{G}(\mathcal{Q}(K))$ by composing the inverse of a closed loop representing an element of $\mathcal{G}(\mathcal{Q}(K))$ by the left.
For each $r \in Z$, $\widetilde{r}$ denotes a lift of a representative path of $r$ satisfying $\widetilde{r}(0) = \widetilde{p}$.

For each $r \in Z$ and $x, y \in \mathcal{Q}(K)$, we define a $3$-dimensional singular chain $C_{r, x, y}$ of $S^{3} \setminus K$ as
\begin{eqnarray*}
 C_{r, x, y} & = & \{ \widetilde{p}, \widetilde{r}(1), x_{\infty}, y_{\infty} \} + \{ \widetilde{p}, \widetilde{r \cdot x}(1), y_{\infty}, x_{\infty} \} \\
 & & \hskip 0.5em + \{ \widetilde{p}, \widetilde{r \cdot x y}(1), (x \ast y)_{\infty}, y_{\infty} \} + \{ \widetilde{p}, \widetilde{r \cdot y}(1), y_{\infty}, (x \ast y)_{\infty} \},
\end{eqnarray*}
where $\{ v_{0}, v_{1}, v_{2}, v_{3} \}$ ($v_{0}, v_{1}, v_{2}, v_{3} \in \overline{\mathbb{H}^{3}}$) denotes a singular simplex of $S^{3} \setminus K$ defined as a map from the tetrahedron $\Delta^{3}$ possibly with ideal vertices to the image of a geodesic tetrahedron spun by $v_{0}, v_{1}, v_{2}$, and $v_{3}$ by $\Psi$, under the assumption that ideal vertices are properly understood.
Further we define a map
\[
 \mathrm{vol} : Z \times \mathcal{Q}(K) \times \mathcal{Q}(K) \longrightarrow \mathbb{R}
\]
by
\begin{eqnarray*}
 \mathrm{vol}(r, x, y) & = & \mathrm{algvol}(\{ \widetilde{p}, \widetilde{r}(1), x_{\infty}, y_{\infty} \}) \\
 & & \hskip 1em + \mathrm{algvol}(\{ \widetilde{p}, \widetilde{r \cdot x}(1), y_{\infty}, x_{\infty} \}) \\
 & & \hskip 3em + \mathrm{algvol}(\{ \widetilde{p}, \widetilde{r \cdot x y}(1), (x \ast y)_{\infty}, y_{\infty} \}) \\
 & & \hskip 5em + \mathrm{algvol}(\{ \widetilde{p}, \widetilde{r \cdot y}(1), y_{\infty}, (x \ast y)_{\infty} \}),
\end{eqnarray*}
where $\mathrm{algvol}(\{ v_{0}, v_{1}, v_{2}, v_{3} \})$ denotes the algebraic volume of a singular simplex $\{ v_{0}, v_{1}, v_{2}, v_{3} \}$.
It is convenient that we extend the domain of $\mathrm{algvol}(\cdot)$ to singular chains linearly.
Then $\mathrm{vol}(r, x, y) = \mathrm{algvol}(C_{r, x, y})$.

\begin{proposition}\label{prop:2_cocycle}
$\mathrm{vol}$ is an $\mathbb{R}$-valued quandle $2$-cocycle with respect to $\mathcal{Q}(K)$ and $Z$.
\end{proposition}

\begin{proof}
For each $r \in Z$ and $x \in \mathcal{Q}(K)$, it is obvious that each simplex constructing a singular chain $C_{r, x, x}$ degenerates.
Thus
\[
 \mathrm{vol}(r, x, x) = \mathrm{algvol}(C_{r, x, x}) = 0.
\]

For each $r \in Z$ and $x, y, z \in \mathcal{Q}(K)$, it is routine to check that several Pachner moves transform a singular chain $(C_{r, x, y} + C_{r \cdot y, x \ast y, z} + C_{r, y, z})$ into another singular chain $(C_{r \cdot x, y, z} + C_{r, x, z} + C_{r \cdot z, x \ast z, y \ast z})$.
Thus
\begin{eqnarray*}
 & & \mathrm{vol}(r, x, y) + \mathrm{vol}(r \cdot y, x \ast y, z) + \mathrm{vol}(r, y, z) \\
 & & \hskip 1em - \mathrm{vol}(r \cdot x, y, z) - \mathrm{vol}(r, x, z) - \mathrm{vol}(r \cdot z, x \ast z, y \ast z) \\
 & = & \mathrm{algvol}(C_{r, x, y} + C_{r \cdot y, x \ast y, z} + C_{r, y, z}) \\
 & & \hskip 1em - \mathrm{algvol}(C_{r \cdot x, y, z} + C_{r, x, z} + C_{r \cdot z, x \ast z, y \ast z}) \\
 & = & 0.
\end{eqnarray*}
\end{proof}

\begin{proposition}\label{prop:2_cocycle_P}
For each shadow coloring $\mathcal{S}$ of a diagram $D$ of $K$ with respect to $\mathcal{Q}(K)$ and $Z$, there exists $k \in \mathbb{Z}$ such that
\[
 \Phi(\mathcal{S}, \mathrm{vol}) = k \cdot \mathrm{vol}(S^{3} \setminus K).
\]
\end{proposition}

\noindent
Here $\mathrm{vol}(S^{3} \setminus K)$ denotes the hyperbolic volume of $S^{3} \setminus K$.

\begin{proof}
Since 
\[
 \Phi(\mathcal{S}, \mathrm{vol}) = \mathrm{algvol} \biggl( \hskip 0.1em \sum_{c \in \mathcal{C}} \varepsilon(c) C_{r, x, y} \hskip 0.1em \biggr),
\]
where $r \in Z$ and $x, y \in \mathcal{Q}(K)$ denote colors around a crossing point $c$, it is sufficient to prove that
\[
 \partial \biggl( \hskip 0.1em \sum_{c \in \mathcal{C}} \varepsilon(c) C_{r, x, y} \hskip 0.1em \biggr) = 0,
\]
that is,
\[
 \biggl( \hskip 0.1em \overline{\sum_{c \in \mathcal{C}} \varepsilon(c) C_{r, x, y}}, \partial \overline{\sum_{c \in \mathcal{C}} \varepsilon(c) C_{r, x, y}} \hskip 0.1em \biggr) \in \mathrm{Z}_{3} \bigl( \hskip 0.1em \overline{S^{3} \setminus K}, \partial \overline{S^{3} \setminus K} \hskip 0.1em \bigr).
\]
Here $\partial(\cdot)$ denotes the boundary operator, $\overline{S^{3} \setminus K}$ a compactification of $S^{3} \setminus K$ with a torus boundary, and $\displaystyle \biggl( \hskip 0.1em \overline{\sum_{c \in \mathcal{C}} \varepsilon(c) C_{r, x, y}}, \partial \overline{\sum_{c \in \mathcal{C}} \varepsilon(c) C_{r, x, y}} \hskip 0.1em \biggr)$ a relative singular chain with respect to a singular chain $\displaystyle \sum_{c \in \mathcal{C}} \varepsilon(c) C_{r, x, y}$ which is naturally defined by the compactification.
At this time,
\[
 \biggl[ \hskip 0.1em \overline{\sum_{c \in \mathcal{C}} \varepsilon(c) C_{r, x, y}}, \partial \overline{\sum_{c \in \mathcal{C}} \varepsilon(c) C_{r, x, y}} \hskip 0.1em \biggr] \in \mathrm{H}_{3} \bigl( \hskip 0.1em \overline{S^{3} \setminus K}, \partial \overline{S^{3} \setminus K} \hskip 0.1em \bigr)
\]
must be an integral multiple of the fundamental class, and thus
\[
 \mathrm{algvol} \biggl( \hskip 0.1em \sum_{c \in \mathcal{C}} \varepsilon(c) C_{r, x, y} \hskip 0.1em \biggr) = k \cdot \mathrm{vol}(S^{3} \setminus K).
\]

By a straightforward calculation,
\begin{eqnarray*}
 \partial C_{r, x, y} = & \hskip -0.5em & \hskip -0.5em \{ \widetilde{p}, \widetilde{r}(1), y_{\infty} \} - \{ \widetilde{p}, \widetilde{r \cdot y}(1), y_{\infty} \} \\
 & \hskip -0.5em + & \hskip -0.5em \{ \widetilde{p}, \widetilde{r \cdot x}(1), x_{\infty} \} - \{ \widetilde{p}, \widetilde{r}(1), x_{\infty} \} \\
 & \hskip -0.5em + & \hskip -0.5em \{ \widetilde{p}, \widetilde{r \cdot x y}(1), y_{\infty} \} - \{ \widetilde{p}, \widetilde{r \cdot x}(1), y_{\infty} \} \\
 & \hskip -0.5em + & \hskip -0.5em \{ \widetilde{p}, \widetilde{r \cdot y}(1), (x \ast y)_{\infty} \} - \{ \widetilde{p}, \widetilde{r \cdot x y}(1), (x \ast y)_{\infty} \},
\end{eqnarray*}
where $\{ v_{0}, v_{1}, v_{2} \}$ ($v_{0}, v_{1}, v_{2} \in \overline{\mathbb{H}^{3}}$) denotes a $2$-dimensional singular simplex of $S^{3} \setminus K$ defined as a map from the triangle $\Delta^{2}$ possibly with ideal vertices to the image of a geodesic triangle spun by $v_{0}, v_{1}$, and $v_{2}$ by $\Psi$.
Further corresponding to each adjacent crossing points $c$ and $c^{\prime}$, singular chains $\varepsilon(c) \partial C_{r, x, y}$ and $\varepsilon(c^{\prime}) \partial C_{r^{\prime}, x^{\prime}, y^{\prime}}$, where $r^{\prime} \in Z$ and $x^{\prime}, y^{\prime} \in \mathcal{Q}(K)$ denote colors around $c^{\prime}$, have canceling terms as depicted in Figure \ref{fig:cancelling_terms}.
Thus
\[
 \partial \biggl( \hskip 0.1em \sum_{c \in \mathcal{C}} \varepsilon(c) C_{r, x, y} \hskip 0.1em \biggr) = \sum_{c \in \mathcal{C}} \varepsilon(c) \partial C_{r, x, y} = 0.
\]
\end{proof}

\begin{figure}[htb]
\begin{center}
\includegraphics[scale=0.43]{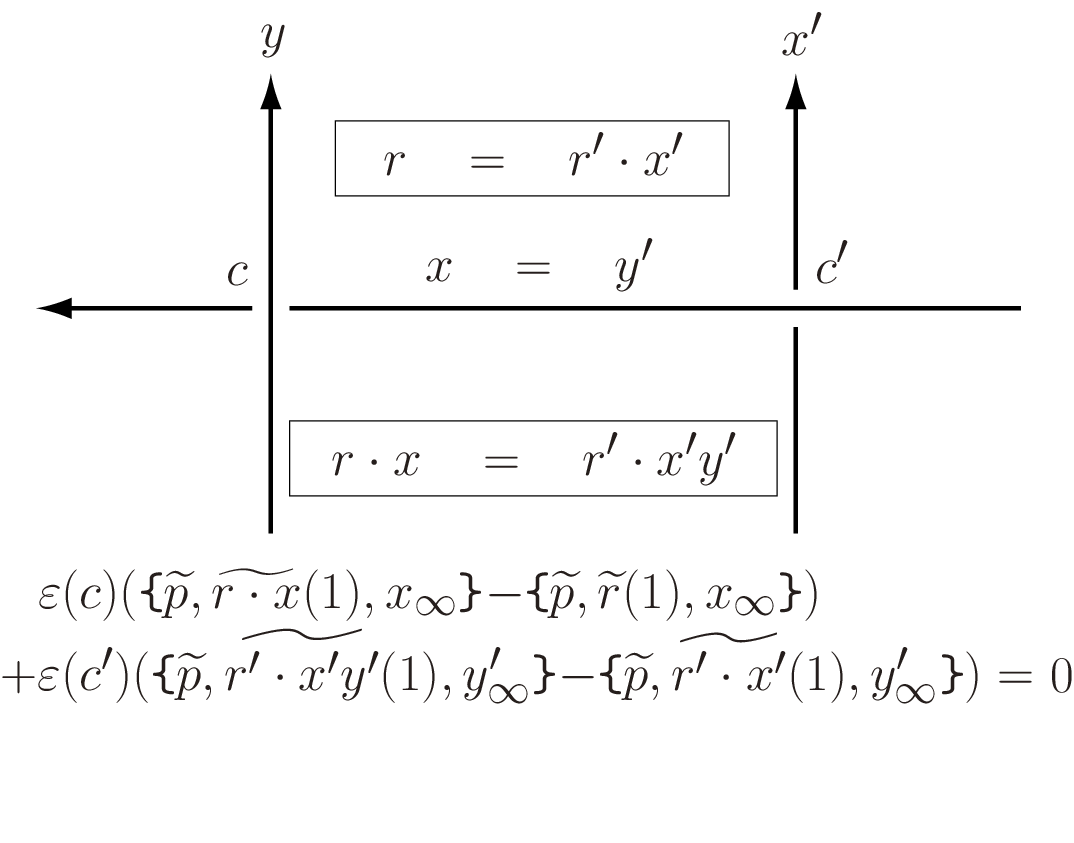}
\end{center}
\vspace{-32pt}
\caption{Adjacent crossing points cancel a pair of terms}
\label{fig:cancelling_terms}
\end{figure}

Furthermore, we refine Proposition \ref{prop:2_cocycle_P} as follows.

\begin{theorem}\label{thm:2_cocycle}
For each shadow coloring $\mathcal{S}$ of a diagram $D$ of $K$ with respect to $\mathcal{Q}(K)$ and $Z$, there exists $k \in \{ -1, 0, 1 \}$ such that
\[
 \Phi(\mathcal{S}, \mathrm{vol}) = k \cdot \mathrm{vol}(S^{3} \setminus K).
\]
\end{theorem}

\noindent
To prove the theorem, we consider the following decomposition of $S^{3} \setminus K$ introduced in \cite{Hatakenaka1}.
Put a diagram $D$ of $K$ on $S^{2}$ which divides $S^{3}$ into two connected components containing $p$ or $q$ respectively.
Take a dual graph of $D$ on $S^{2}$, and consider its suspension with respect to $p$ and $q$.
Then we have a decomposition of $S^{3} \setminus K$ into thin regions like bananas illustrated in Figure \ref{fig:banana}.
Further we cut each banana into four pieces as depicted in Figure \ref{fig:cut_banana}.

\begin{figure}[htb]
\begin{center}
\includegraphics[scale=0.43]{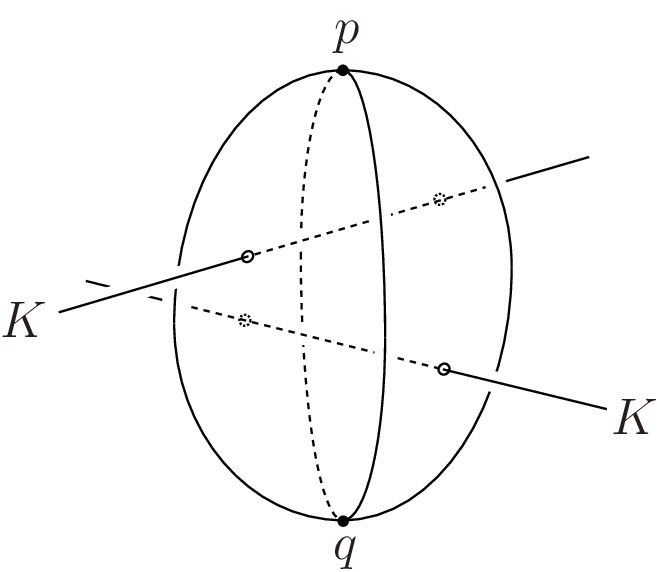}
\end{center}
\caption{A banana}
\label{fig:banana}
\end{figure}

\begin{figure}[htb]
\begin{center}
\includegraphics[scale=0.43]{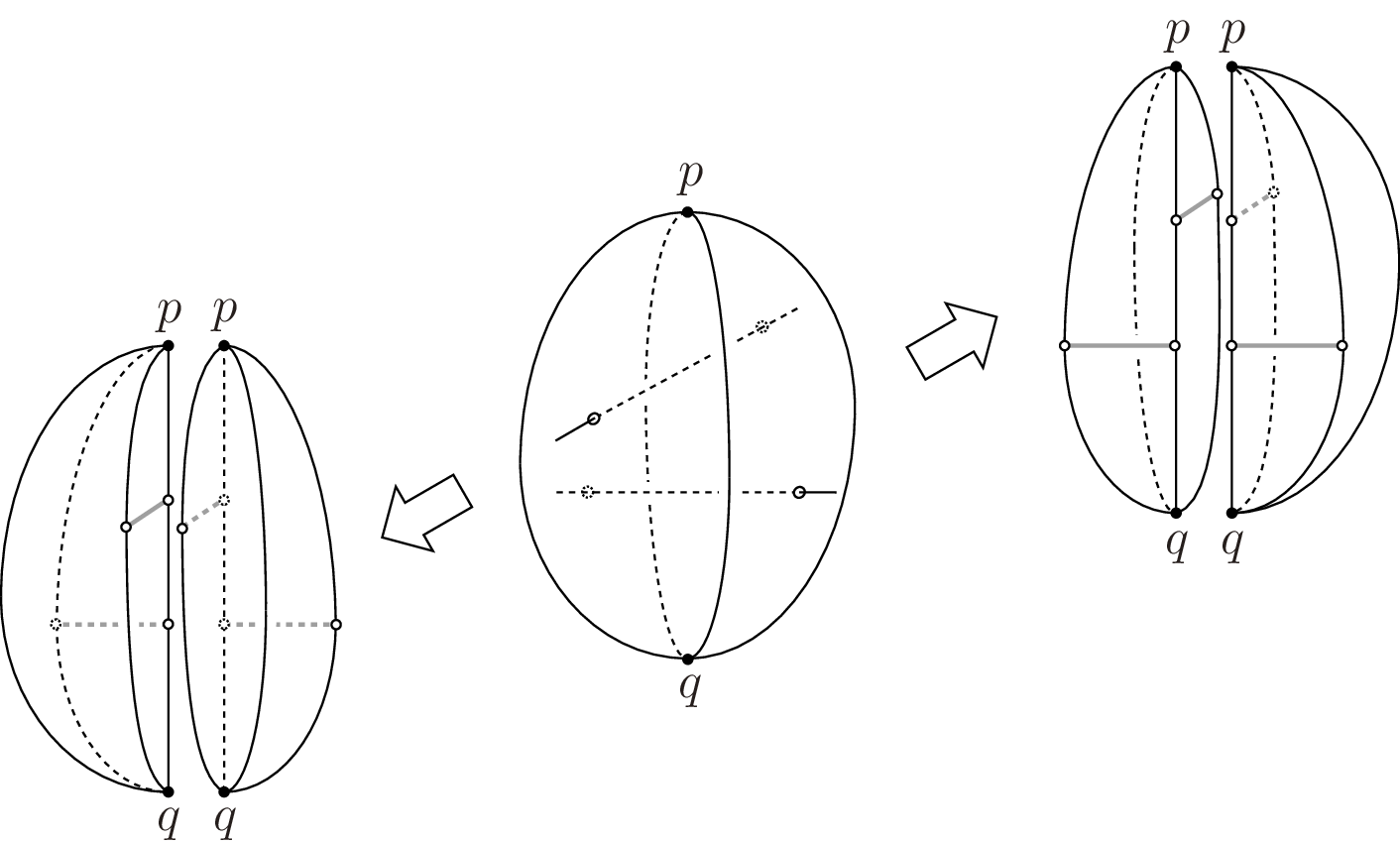}
\end{center}
\caption{Cutting a banana into four pieces}
\label{fig:cut_banana}
\end{figure}

\begin{proof}[Proof of Theorem \ref{thm:2_cocycle}]
For each piece $P$ of bananas, define a surjective continuous map $\tau$ from $P$ to the tetrahedron $\Delta^{3}$ with two ideal vertices as depicted in Figure \ref{fig:banana2tetrahedronA} or \ref{fig:banana2tetrahedronB} depending on the shape of $P$.
Let $x$ or $y$ in $\mathcal{Q}(K)$ be the color of the arc $a$ or $b$ respectively, and $r \in Z$ the color of the region with which the edge $e$ intersects.
Then the composition $\{ \widetilde{p}, \widetilde{r}(1), x_{\infty}, y_{\infty} \} \circ \tau$ is a continuous map from $P$ to $S^{3} \setminus K$.
By the construction, modifying each $\tau$ step-by-step if necessary, we can assume that
\[
 \{ \widetilde{p}, \widetilde{r}(1), x_{\infty}, y_{\infty} \} \circ \tau|_{\partial P \cap \partial P^{\prime}} = \{ \widetilde{p}, \widetilde{r^{\prime}}(1), x^{\prime}_{\infty}, y^{\prime}_{\infty} \} \circ \tau^{\prime}|_{\partial P \cap \partial P^{\prime}}
\]
for each pair of pieces $P$ and $P^{\prime}$ of bananas, where $\tau^{\prime}$ denotes a surjective continuous map from $P^{\prime}$ to $\Delta^{3}$, and $x^{\prime}, y^{\prime} \in \mathcal{Q}(K)$ and $r^{\prime} \in Z$ colors with respect to $P^{\prime}$.
Thus we have a continuous map
\[
 f_{\mathcal{S}} : S^{3} \setminus K \longrightarrow S^{3} \setminus K
\]
satisfying $f_{\mathcal{S}}|_{P} = \{ \widetilde{p}, \widetilde{r}(1), x_{\infty}, y_{\infty} \} \circ \tau$.
By the assumption that $K$ is hyperbolic, the degree of $f_{\mathcal{S}}$ must be $1$, $-1$ or $0$, or else the simplicial volume of $S^{3} \setminus K$ must be $0$ being untrue to the assumption.
Thus
\[
 \biggl[ \hskip 0.1em \overline{\sum_{c \in \mathcal{C}} \varepsilon(c) C_{r, x, y}}, \partial \overline{\sum_{c \in \mathcal{C}} \varepsilon(c) C_{r, x, y}} \hskip 0.1em \biggr] \in \mathrm{H}_{3} \bigl( \hskip 0.1em \overline{S^{3} \setminus K}, \partial \overline{S^{3} \setminus K} \hskip 0.1em \bigr)
\]
must be $1$, $-1$ or $0$ times the fundamental class.
\end{proof}

\begin{figure}[htb]
\begin{center}
\includegraphics[scale=0.43]{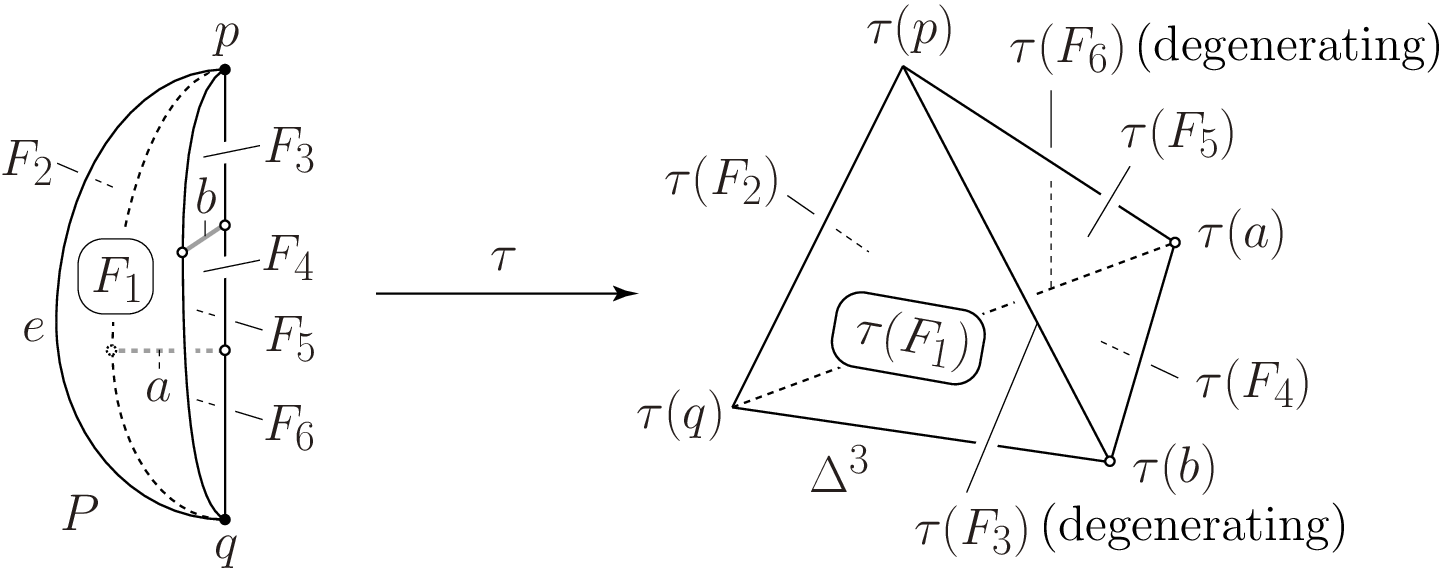}
\end{center}
\caption{A surjective continuous map $\tau$}
\label{fig:banana2tetrahedronA}
\end{figure}

\begin{figure}[htb]
\begin{center}
\includegraphics[scale=0.43]{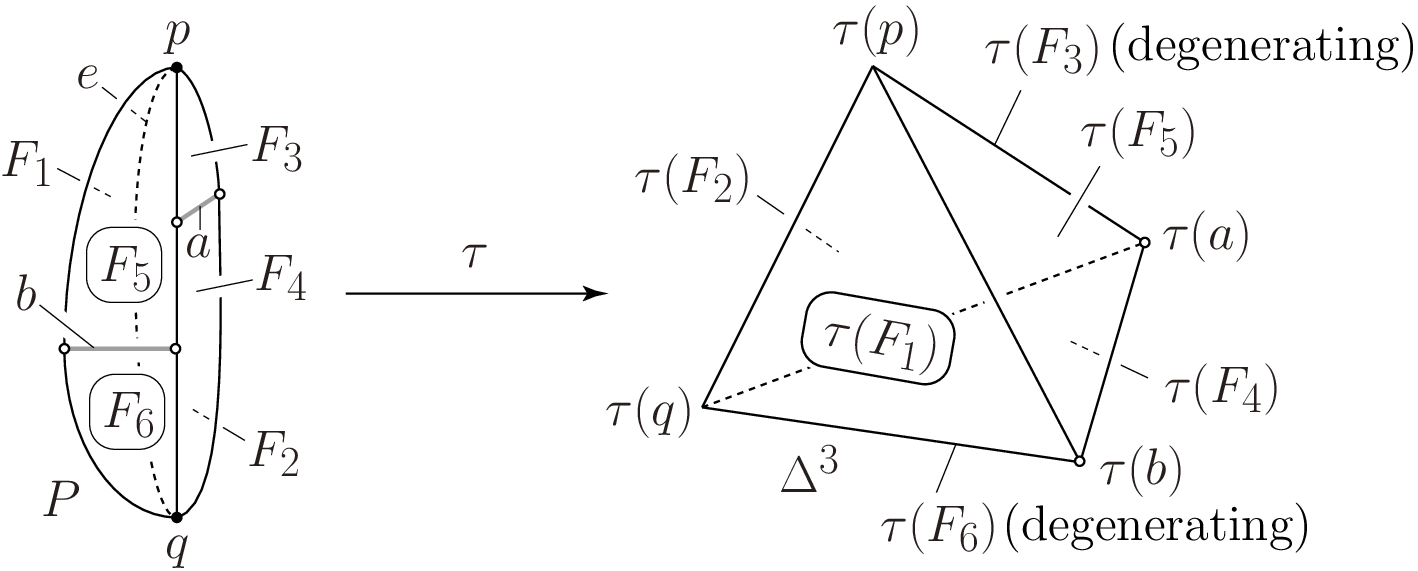}
\end{center}
\caption{Another surjective continuous map $\tau$}
\label{fig:banana2tetrahedronB}
\end{figure}

Let $\mathcal{A}_{\ast}$ be an arc coloring of a diagram $D$ of $K$ with respect to $\mathcal{Q}(K)$ mapping each arc of $D$ to the Wirtinger generator (cf. Section 3.D. of \cite{Rolfsen}) with respect to the arc, $\mathcal{R}_{\ast}$ a region coloring of $D$ with respect to $Z$ mapping each region to the homotopy class of the edge of bananas which intersects with the region, and $\mathcal{S}_{\ast} = (\mathcal{A}_{\ast}, \mathcal{R}_{\ast})$.

\begin{theorem}\label{thm:natural_coloring_2_cocycle}
\[
 \Phi(\mathcal{S}_{\ast}, \mathrm{vol}) = \mathrm{vol}(S^{3} \setminus K).
\]
\end{theorem}

\begin{proof}
For each piece $P$ of bananas, $\{ \widetilde{p}, \widetilde{r}(1), x_{\infty}, y_{\infty} \} \circ \tau$ is homotopic to the identity map of $P$, because $\{ \widetilde{p}, \widetilde{r}(1), x_{\infty}, y_{\infty} \} \circ \tau$ corresponds to the straightening of the identity map.
Thus $f_{\mathcal{S}_{\ast}}$ is homotopic to the identity map of $S^{3} \setminus K$.
\end{proof}

\section{Determining invertibility and amphicheirality}\label{sec:determining_invertibility_and_amphicheirality}

Let $K$ be an oriented hyperbolic knot in $S^{3}$.
We denote $K$ with reversed orientation by $-K$, and a mirror image of $K$ by $K^{\ast}$.
$-K^{\ast}$ is thus a mirror image of $K$ with reversed orientation.

\begin{theorem}\label{thm:negative_amphichirality}
$K$ is equivalent to $-K^{\ast}$ if and only if there exists a shadow coloring $\mathcal{S}$ of a diagram $D$ of $K$ with respect to $\mathcal{Q}(K)$ and $Z$ satisfying
\[
 \Phi(\mathcal{S}, \mathrm{vol}) = - \mathrm{vol}(S^{3} \setminus K).
\]
\end{theorem}

\begin{proof}
First, we show ``if'' part.
By Thurston's rigidity theorem (cf. \cite{Gromov1} for example), there is an orientation reversing homeomorphism $f$ of $S^{3} \setminus K$ being homotopic to $f_{\mathcal{S}}$.
Since $f_{\mathcal{S}}$ maps each positive meridian of $K$ to a positive meridian, we can extend $f$ to an orientation reversing homomorphism of $(S^{3}, K)$ which reverses the orientation of $K$.
Suppose
\[
 m : (S^{3}, K) \longrightarrow (S^{3}, K^{\ast})
\]
is a mirroring.
Then $m \circ f$ is an orientation preserving homeomorphism which reverses the orientation of $K$.
Thus $K$ is equivalent to $-K^{\ast}$.

Next, we show ``only if'' part.
By the assumption, there exists an orientation preserving homeomorphism
\[
 g : (S^{3}, K) \longrightarrow (S^{3}, K^{\ast})
\]
which reverses the orientation of $K$, and thus the composition $m^{-1} \circ g$ is an orientation reversing homeomorphism of $(S^{3}, K)$ which reverses the orientation of $K$.
$m^{-1} \circ g$ induces a map mapping a shadow coloring $\mathcal{S}^{\prime}$ of $D$ with respect to $\mathcal{Q}(K)$ and $Z$ to a shadow coloring $\mathcal{S}$ of $D$ with respect to $\mathcal{Q}(K)$ and $Z$.
We assume $\Phi(\mathcal{S}^{\prime}, \mathrm{vol}) = \mathrm{vol}(S^{3} \setminus K)$.
Further $m^{-1} \circ g$ induces a chain map mapping a singular chain $\displaystyle \sum_{c \in \mathcal{C}} \varepsilon(c) C_{r^{\prime}, x^{\prime}, y^{\prime}}$ to a singular chain $\displaystyle \sum_{c \in \mathcal{C}} \varepsilon(c) C_{r, x, y}$, where $x^{\prime}, y^{\prime} \in \mathcal{Q}(K)$ and $r^{\prime} \in Z$ denote colors around a crossing point $c$ with respect to $\mathcal{S}^{\prime}$, and the correspondence of $x^{\prime}, y^{\prime} \in \mathcal{Q}(K)$ and $r^{\prime} \in Z$ with $x, y \in \mathcal{Q}(K)$ and $r \in Z$ is induced by $m^{-1} \circ g$.
Since $m^{-1} \circ g$ reverses the orientation of $S^{3}$,
\[
 \biggl[ \hskip 0.1em \overline{\sum_{c \in \mathcal{C}} \varepsilon(c) C_{r, x, y}}, \partial \overline{\sum_{c \in \mathcal{C}} \varepsilon(c) C_{r, x, y}} \hskip 0.1em \biggr] \in \mathrm{H}_{3} \bigl( \hskip 0.1em \overline{S^{3} \setminus K}, \partial \overline{S^{3} \setminus K} \hskip 0.1em \bigr)
\]
is $-1$ times the fundamental class.
\end{proof}

Now, let us consider another shadow coloring $\mathcal{S}$ of a diagram $D$ of $K$ with respect to $\mathcal{Q}(K^{\prime})$ and $Z^{\prime}$, where $K^{\prime}$ is another oriented hyperbolic knot in $S^{3}$, and $Z^{\prime}$ the set of homotopy classes of paths in $S^{3} \setminus K^{\prime}$.
Then we also obtain a continuous map
\[
 f_{\mathcal{S}} : S^{3} \setminus K \longrightarrow S^{3} \setminus K^{\prime}
\]
by the same construction described in the previous section, although the range does not coincide with the domain.
Further it is easy to see that
\begin{eqnarray*}
 \Phi(\mathcal{S}, \mathrm{vol}) & = & \mathrm{algvol}(f_{\mathcal{S}}(S^{3} \setminus K)) \\
 & = & k \cdot \mathrm{vol}(S^{3} \setminus K^{\prime}) \qquad (k \in \mathbb{Z}).
\end{eqnarray*}
In particular, the following theorem holds.

\begin{theorem}\label{thm:orientation_reversed}
For each shadow coloring $\mathcal{S}$ of a diagram $D$ of $K$ with respect to $\mathcal{Q}(-K)$ and $Z$, there exists $k \in \{ -1, 0, 1 \}$ such that
\[
 \Phi(\mathcal{S}, \mathrm{vol}) = k \cdot \mathrm{vol}(S^{3} \setminus K).
\]
\end{theorem}

\noindent
We omit the proof.

\begin{theorem}\label{thm:invertibility}
$K$ is equivalent to $-K$ if and only if there exists a shadow coloring $\mathcal{S}$ of a diagram $D$ of $K$ with respect to $\mathcal{Q}(-K)$ and $Z$ satisfying
\[
 \Phi(\mathcal{S}, \mathrm{vol}) = \mathrm{vol}(S^{3} \setminus K).
\]
\end{theorem}

\begin{theorem}\label{thm:positive_amphichirality}
$K$ is equivalent to $K^{\ast}$ if and only if there exists a shadow coloring $\mathcal{S}$ of a diagram $D$ of $K$ with respect to $\mathcal{Q}(-K)$ and $Z$ satisfying
\[
 \Phi(\mathcal{S}, \mathrm{vol}) = - \mathrm{vol}(S^{3} \setminus K).
\]
\end{theorem}

\noindent
Theorem \ref{thm:positive_amphichirality} can be proved by the same line along the proof of Theorem \ref{thm:invertibility}.
Thus we only prove Theorem \ref{thm:invertibility}.

\begin{proof}[Proof of Theorem \ref{thm:invertibility}]
First, we show ``if'' part.
By Thurston's rigidity theorem again, there is an orientation preserving homeomorphism $f$ of $S^{3} \setminus K$ being homotopic to $f_{\mathcal{S}}$.
Since $f_{\mathcal{S}}$ maps each positive meridian of $K$ into a positive meridian of $-K$, we can also extend $f$ to an orientation preserving homomorphism of $(S^{3}, K)$ which reverses the orientation of $K$.
Thus $K$ is equivalent to $-K$.

Next, we show ``only if'' part.
By the assumption, there exists an orientation preserving homeomorphism
\[
 g : (S^{3}, K) \longrightarrow (S^{3}, K)
\]
which reverses the orientation of $K$.
$g$ induces a map mapping a shadow coloring $\mathcal{S}^{\prime}$ of $D$ with respect to $\mathcal{Q}(K)$ and $Z$ to a shadow coloring $\mathcal{S}$ of $D$ with respect to $\mathcal{Q}(-K)$ and $Z$.
We assume $\Phi(\mathcal{S}^{\prime}, \mathrm{vol}) = \mathrm{vol}(S^{3} \setminus K)$.
Further $g$ induces a chain map mapping a singular chain $\displaystyle \sum_{c \in \mathcal{C}} \varepsilon(c) C_{r^{\prime}, x^{\prime}, y^{\prime}}$ to a singular chain $\displaystyle \sum_{c \in \mathcal{C}} \varepsilon(c) C_{r, x, y}$, where $x^{\prime}, y^{\prime} \in \mathcal{Q}(K)$ and $r^{\prime} \in Z$ denote colors around a crossing point $c$ with respect to $\mathcal{S}^{\prime}$, and the correspondence of $x^{\prime}, y^{\prime} \in \mathcal{Q}(K)$ and $r^{\prime} \in Z$ with $x, y \in \mathcal{Q}(-K)$ and $r \in Z$ is induced by $g$.
Since $g$ preserves the orientation of $S^{3}$,
\[
 \biggl[ \hskip 0.1em \overline{\sum_{c \in \mathcal{C}} \varepsilon(c) C_{r, x, y}}, \partial \overline{\sum_{c \in \mathcal{C}} \varepsilon(c) C_{r, x, y}} \hskip 0.1em \biggr] \in \mathrm{H}_{3} \bigl( \hskip 0.1em \overline{S^{3} \setminus K}, \partial \overline{S^{3} \setminus K} \hskip 0.1em \bigr)
\]
is the fundamental class.
\end{proof}

\section{Example}\label{sec:example}

Let $K$ be an oriented hyperbolic knot in $S^{3}$, and $\mathcal{S}$ a shadow coloring of a diagram $D$ of $K$ with respect to $\mathcal{Q}(K)$ and $Z$.
We remark that even if we relocate the point $p$ or $q$ to another point $p^{\prime}$ or $q^{\prime}$ in $S^{3}$ along a path $\gamma$ or $\delta$ respectively, the homology class of a singular chain $\displaystyle \sum_{c \in \mathcal{C}} \varepsilon(c) C_{r, x, y}$ with respect to $\mathcal{S}$ does not change.
Further if we choose $p^{\prime}$ or $q^{\prime}$ on $K$ then there is a positive meridian $w$ or $z$ in $\mathcal{Q}(K)$ satisfying $\widetilde{\gamma}(1) = w_{\infty}$ with $\widetilde{\gamma}(0) = \widetilde{p}$ or $\widetilde{r \cdot \delta}(1) = z_{\infty}$ with $\widetilde{r \cdot \delta}(0) = \widetilde{p}$ respectively, where $r \cdot \delta$ denotes the composition of a representative path of $r$ and $\delta$, and each singular chain $C_{r, x, y}$ changes into a singular chain
\begin{eqnarray*}
 C^{w}_{z, x, y} & = & \{ w_{\infty}, z_{\infty}, x_{\infty}, y_{\infty} \} + \{ w_{\infty}, (z \ast x)_{\infty}, y_{\infty}, x_{\infty} \} \\
 & & \hskip 1em + \{ w_{\infty}, ((z \ast x) \ast y)_{\infty}, (x \ast y)_{\infty}, y_{\infty} \} \\
 & & \hskip 2em + \{ w_{\infty}, (z \ast y)_{\infty}, y_{\infty}, (x \ast y)_{\infty} \}.
\end{eqnarray*}
Thus the following theorem holds with a map
\[
 \mathrm{vol}^{w} : \mathcal{Q}(K) \times \mathcal{Q}(K) \times \mathcal{Q}(K) \longrightarrow \mathbb{R}
\]
defined by $\mathrm{vol}^{w}(z, x, y) = \mathrm{algvol}(C^{w}_{z, x, y})$ with some $w \in \mathcal{Q}(K)$.

\begin{theorem}\label{thm:3_cocycle}
$\mathrm{vol}^{w}$ is an $\mathbb{R}$-valued quandle $2$-cocycle with respect to $\mathcal{Q}(K)$.
Further for each shadow coloring $\mathcal{S}$ of a diagram $D$ of $K$ with respect to $\mathcal{Q}(K)$, there exists $k \in \{ -1, 0, 1 \}$ such that
\[
 \Phi(\mathcal{S}, \mathrm{vol}^{w}) = k \cdot \mathrm{vol}(S^{3} \setminus K).
\]
\end{theorem}

\noindent
We omit the proof.
It is easy to see that similar theorems to Theorem \ref{thm:negative_amphichirality}, \ref{thm:orientation_reversed}, \ref{thm:invertibility} and \ref{thm:positive_amphichirality} also hold with respect to this quandle $2$-cocycle.

We close this paper by computing some elements of above quandle cocycle invariants for the figure eight knot.
Associated with a diagram $D$ of an oriented figure eight knot $K$, we choose Wirtinger generators of $\pi_{1}(S^{3} \setminus K)$ $x$, $y$, $z$, and $w$ as depicted in Figure \ref{fig:figure_eight}.
Further we define a holonomy representation $\rho$ or $\rho_{-}$ of $\pi_{1}(S^{3} \setminus K)$ or $\pi_{1}(S^{3} \setminus -K)$ to satisfy the following equations respectively:
\[
 \rho(x) = \begin{pmatrix} \frac{1 - \sqrt{-3}}{2} & \frac{1 + \sqrt{-3}}{2} \\ \frac{-1 - \sqrt{-3}}{2} & \frac{3 + \sqrt{-3}}{2} \end{pmatrix}, \
 \rho(y) = \begin{pmatrix} 1 & \frac{1 + \sqrt{-3}}{2} \\ 0 & 1 \end{pmatrix},
\]
\[
 \rho(z) = \begin{pmatrix} \frac{3 + \sqrt{-3}}{2} & \frac{1 - \sqrt{-3}}{2} \\ 1 & \frac{1 - \sqrt{-3}}{2} \end{pmatrix}, \mathrm{and} \enskip
 \rho(w) = \begin{pmatrix} 1 & 0 \\ 1 & 1 \end{pmatrix}.
\]
\[
 \rho_{-}(x^{-1}) = \begin{pmatrix} 1 & 0 \\ 1 & 1 \end{pmatrix}, \
 \rho_{-}(y^{-1}) = \begin{pmatrix} 1 + \sqrt{-3} & \frac{1 - \sqrt{-3}}{2} \\ \frac{3 + 3 \sqrt{-3}}{2} & 1 - \sqrt{-3} \end{pmatrix},
\]
\[
 \rho_{-}(z^{-1}) = \begin{pmatrix} \frac{3 + \sqrt{-3}}{2} & \frac{1 - \sqrt{-3}}{2} \\ 1 & \frac{1 - \sqrt{-3}}{2} \end{pmatrix}, \mathrm{and} \enskip
 \rho_{-}(w^{-1}) = \begin{pmatrix} \frac{1 - \sqrt{-3}}{2} & \frac{1 + \sqrt{-3}}{2} \\ \frac{-1 - \sqrt{-3}}{2} & \frac{3 + \sqrt{-3}}{2} \end{pmatrix}.
\]

\begin{figure}[htb]
\begin{center}
\includegraphics[scale=0.35]{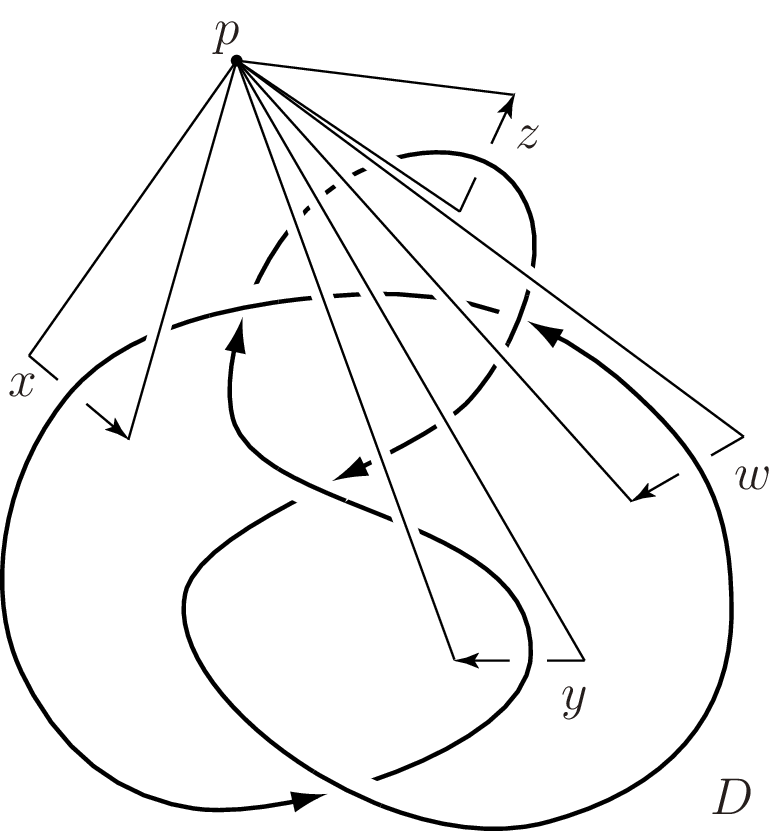}
\end{center}
\caption{Wirtinger generators of $\pi_{1}(S^{3} \setminus K)$}
\label{fig:figure_eight}
\end{figure}

\begin{example}
For a shadow coloring $\mathcal{S}_{1}$ of $D$ with respect to $\mathcal{Q}(K)$ depicted in the left-hand side of Figure \ref{fig:shadow_coloring_of_figure_eight12},
\begin{eqnarray*}
 \Phi(\mathcal{S}_{1}, \mathrm{vol}^{w}) & = & \mathrm{algvol}(\{ w_{\infty}, (y \ast z)_{\infty}, y_{\infty}, z_{\infty} \} - \{ w_{\infty}, y_{\infty}, z_{\infty}, x_{\infty} \}) \\
 & = & \mathrm{algvol}(\{ 0, \tfrac{-1 + \sqrt{-3}}{2}, \infty, \tfrac{1 + \sqrt{-3}}{2} \} - \{ 0, \infty, \tfrac{1 + \sqrt{-3}}{2}, 1 \}) \\
 & = & \mathrm{vol}(S^{3} \setminus K),
\end{eqnarray*}
where we use the upper half-space model of $\mathbb{H}^{3}$.
\end{example}

\begin{example}
For a shadow coloring $\mathcal{S}_{2}$ of $D$ with respect to $\mathcal{Q}(K)$ depicted in the right-hand side of Figure \ref{fig:shadow_coloring_of_figure_eight12},
\begin{eqnarray*}
 \Phi(\mathcal{S}_{2}, \mathrm{vol}^{w}) & = & \mathrm{algvol}(\{ w_{\infty}, (y \ast w)_{\infty}, (y \ast z)_{\infty}, y_{\infty} \} \\
 & & \hskip 6em - \{ w_{\infty}, (y \ast z)_{\infty}, y_{\infty}, z_{\infty} \}) \\
 & = & \mathrm{algvol}(\{ 0, -1, \tfrac{-1 + \sqrt{-3}}{2}, \infty \} - \{ 0, \tfrac{-1 + \sqrt{-3}}{2}, \infty, \tfrac{1 + \sqrt{-3}}{2} \}) \\
 & = & - \mathrm{vol}(S^{3} \setminus K).
\end{eqnarray*}
\end{example}

\begin{figure}[htb]
\begin{center}
\includegraphics[scale=0.35]{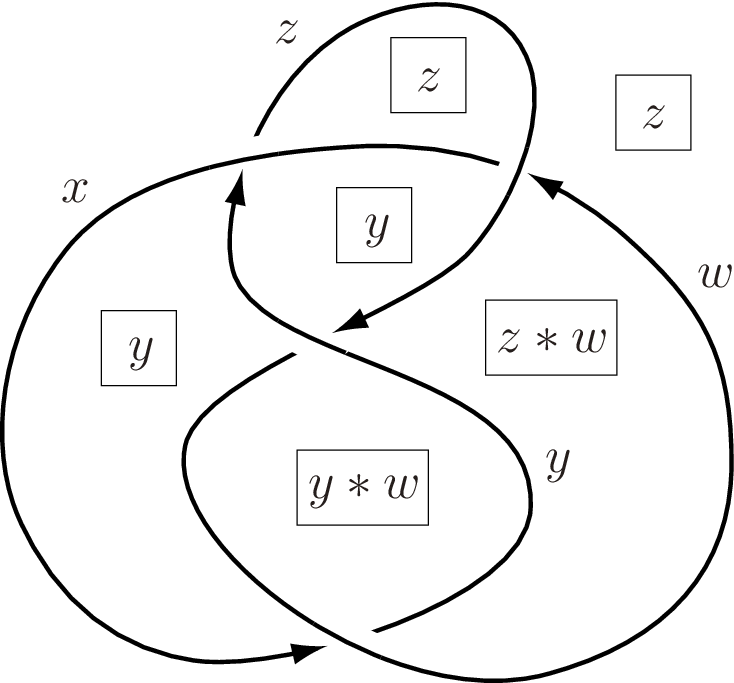} \qquad \qquad
\includegraphics[scale=0.35]{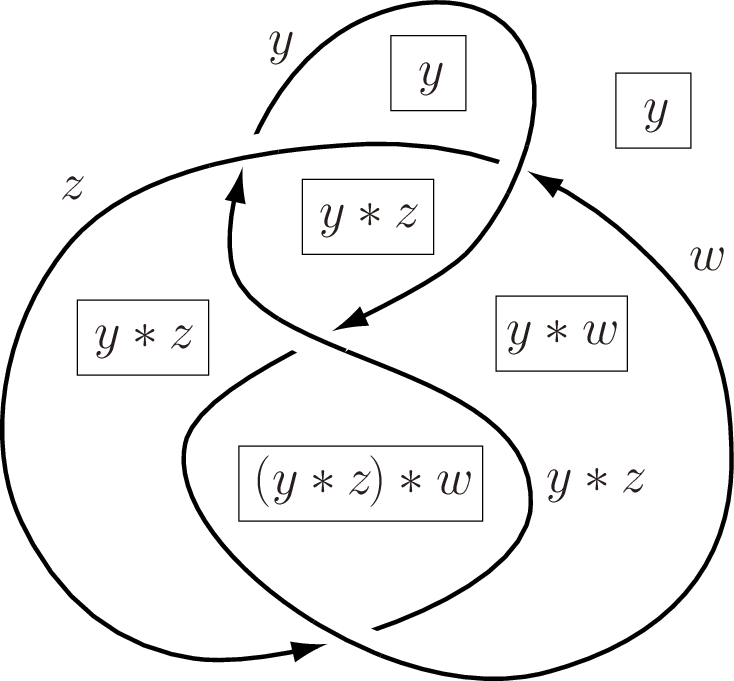}
\end{center}
\caption{$\mathcal{S}_{1}$ (left) and $\mathcal{S}_{2}$ (right)}
\label{fig:shadow_coloring_of_figure_eight12}
\end{figure}

\begin{example}
For a shadow coloring $\mathcal{S}_{3}$ of $D$ with respect to $\mathcal{Q}(-K)$ depicted in the left-hand side of Figure \ref{fig:shadow_coloring_of_figure_eight34},
\begin{eqnarray*}
 \Phi(\mathcal{S}_{3}, \mathrm{vol}^{x^{-1}}) & = & \mathrm{algvol}(\{ x^{-1}_{\infty}, w^{-1}_{\infty}, z^{-1}_{\infty}, (w^{-1} \ast x^{-1})_{\infty} \} \\
 & & \hskip 3em - \{ x^{-1}_{\infty}, (x^{-1} \ast w^{-1})_{\infty}, (w^{-1} \ast x^{-1})_{\infty}, w^{-1}_{\infty} \}) \\
 & = & \mathrm{algvol}(\{ 0, 1, \tfrac{1 + \sqrt{-3}}{2}, \infty \} - \{ 0, \tfrac{1 - \sqrt{-3}}{2}, \infty, 1 \}) \\
 & = & \mathrm{vol}(S^{3} \setminus K).
\end{eqnarray*}
\end{example}

\begin{example}
For a shadow coloring $\mathcal{S}_{4}$ of $D$ with respect to $\mathcal{Q}(-K)$ depicted in the right-hand side of Figure \ref{fig:shadow_coloring_of_figure_eight34},
\begin{eqnarray*}
 \Phi(\mathcal{S}_{4}, \mathrm{vol}^{x^{-1}}) & = & \mathrm{algvol}(\{ x^{-1}_{\infty}, z^{-1}_{\infty}, (w^{-1} \ast x^{-1})_{\infty}, (z^{-1} \ast x^{-1})_{\infty} \} \\
 & & \hskip 4em - \{ x^{-1}_{\infty}, (x^{-1} \ast z^{-1})_{\infty}, (z^{-1} \ast x^{-1})_{\infty}, z^{-1}_{\infty} \}) \\
 & = & \mathrm{algvol}(\{ 0, \tfrac{1 + \sqrt{-3}}{2}, \infty, \tfrac{-1 + \sqrt{-3}}{2} \} \\
 & & \hskip 4em - \{ 0, \tfrac{\sqrt{-3}}{3}, \tfrac{-1 + \sqrt{-3}}{2}, \tfrac{1 + \sqrt{-3}}{2} \}) \\
 & = & - \mathrm{vol}(S^{3} \setminus K).
\end{eqnarray*}
\end{example}

\begin{figure}[htb]
\begin{center}
\includegraphics[scale=0.25]{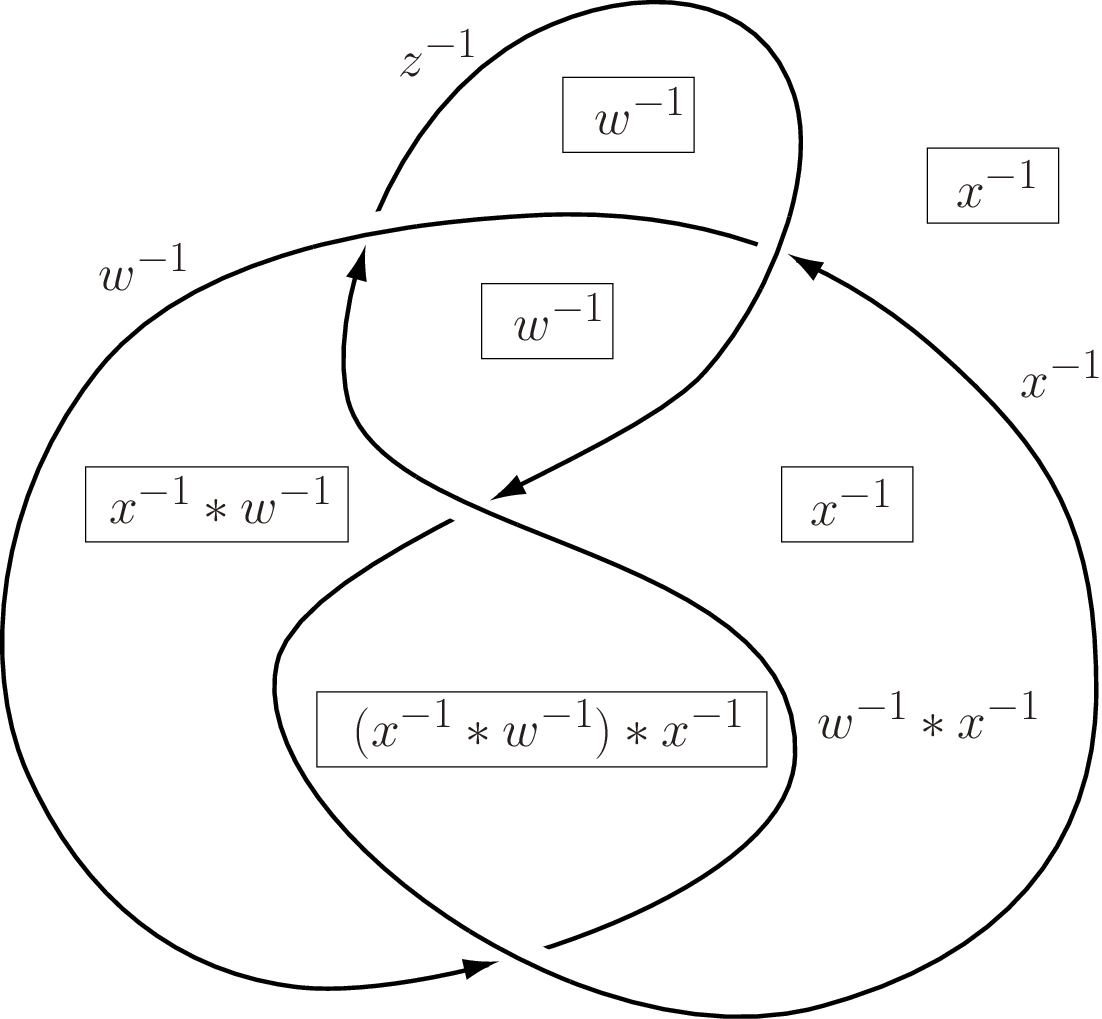} \qquad \qquad
\includegraphics[scale=0.25]{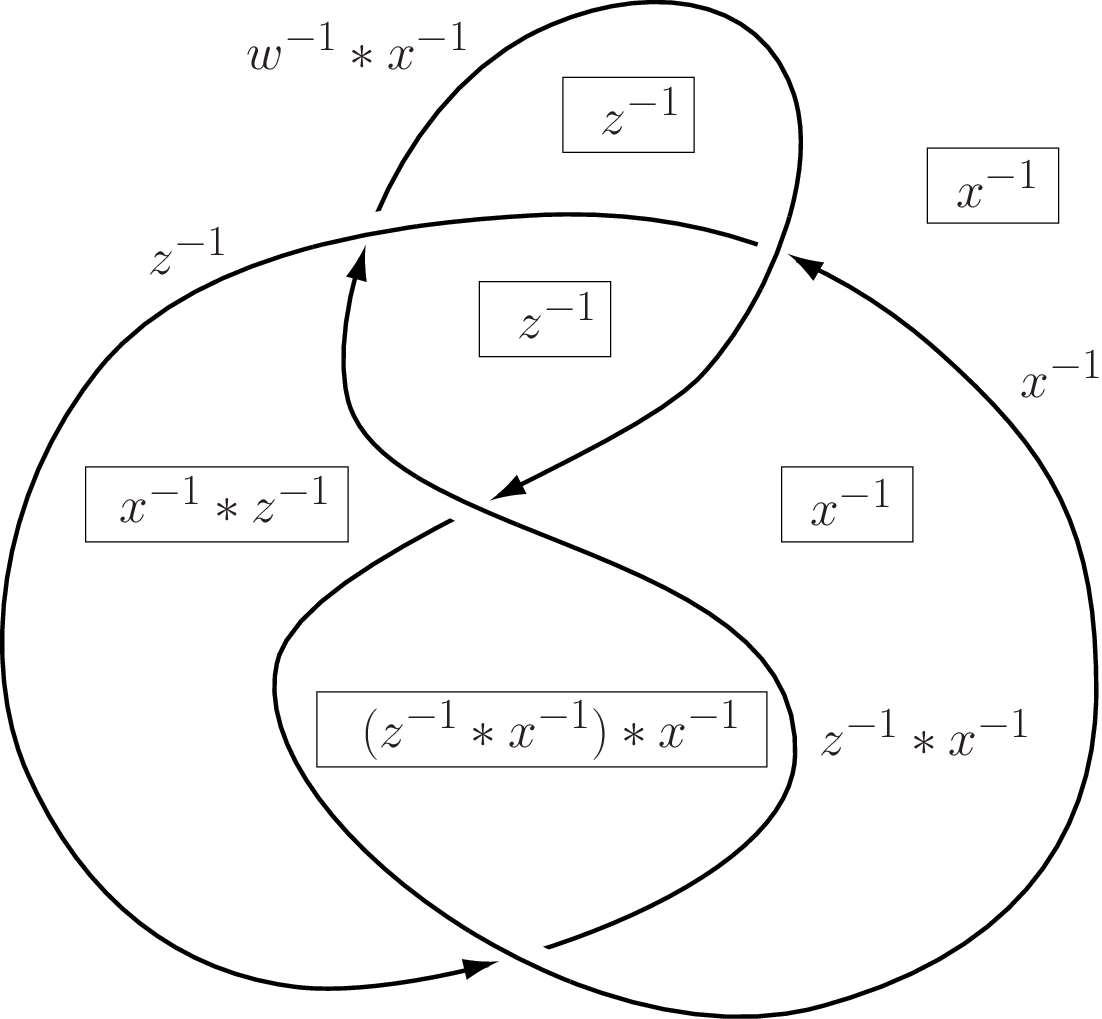}
\end{center}
\caption{$\mathcal{S}_{3}$ (left) and $\mathcal{S}_{4}$ (right)}
\label{fig:shadow_coloring_of_figure_eight34}
\end{figure}

\noindent
In conclusion, we have confirmed that the figure eight knot is invertible and positive/negative amphicheiral, as is well known.

\bibliographystyle{amsplain}

\begin{thebibliography}{99}

\bibitem{CEGS1}
J. S. Carter, M. Elhamdadi, M. Gra\~{n}a and M. Saito, {\it Cocycle knot invariants from quandle modules and generalized quandle homology}, Osaka J. Math. {\bf 42} (2005), 499--541.

\bibitem{CJKLS1}
J. S. Carter, D. Jelsovsky, S. Kamada, L. Langford and M. Saito, {\it Quandle cohomology and state-sum invariants of knotted curves and surfaces}, Trans. Amer. Math. Soc. {\bf 355} (2003), 3947--3989.

\bibitem{Eisermann1}
M. Eisermann, {\it Homological characterization of the unknot}, J. Pure Appl. Algebra {\bf 177} (2003), 131--157.

\bibitem{FR1}
R. Fenn and C. Rourke, {\it Racks and links in codimension two}, J. Knot Theory Ramifications {\bf 1} (1992), 343--406.

\bibitem{FRS1}
R. Fenn, C. Rourke and B. Sanderson, {\it Trunks and classifying spaces}, Appl. Categ. Structures {\bf 3} (1995) 321--356.

\bibitem{FRS2}
R. Fenn, C. Rourke and B. Sanderson, {\it James bundles and applications}, preprint at \url{http://www.maths.warwick.ac.uk/~cpr/}.

\bibitem{Gromov1}
M. Gromov, {\it Hyperbolic manifolds {\upshape (}according to Thurston and J{\o}rgensen{\upshape )}}, Bourbaki Seminar, Vol. 1979/80, pp. 40--53, Lecture Notes in Math., {\bf 842}, Springer, Berlin-New York, 1981.

\bibitem{Hatakenaka1}
E. Hatakenaka, {\it Invariants of $3$-manifolds derived from covering presentations}, preprint.

\bibitem{Joyce1}
D. Joyce, {\it A classifying invariant of knots, the knot quandle}, J. Pure Appl. Algebra {\bf 23} (1982), 37--65.

\bibitem{Kamada1}
S. Kamada, {\it Quandles with good involutions, their homologies and knot invariants}, Intelligence of low dimensional topology 2006, 101--108, Ser. Knots Everything, 40, World Sci. Publ., Hackensack, NJ, 2007.

\bibitem{Matveev1}
S. V. Matveev, {\it Distributive groupoids in Knot theory}, Mat. Sb. (N.S.) {\bf 119(161)} (1982), 78--88 (in Russian).

\bibitem{Rolfsen}
D. Rolfsen, {\it Knots and Links}, Mathematics Lecture Series, No. 7. Publish or Perish, Inc., Berkeley, Calif., 1976.

\bibitem{RS1}
C. Rourke and B. Sanderson, {\it There are two $2$-twist-spun trefoils}, preprint at \url{http://www.maths.warwick.ac.uk/~cpr/}.

\bibitem{Satoh1}
S. Satoh, {\it Surface diagrams of twist-spun $2$-knots}, in ``Knots 2000 Korea, {\bf 1} (Yongpyong)'', J. Knot Theory Ramifications {\bf 11} (2002), 413--430.

\bibitem{SS1}
S. Satoh and A. Shima, {\it The $2$-twist-spun trefoil has the triple point number four}, Trans. Amer. Math. Soc. {\bf 356} (2004), 1007--1024.

\end{thebibliography}

\end{document}